\documentclass[12pt,a4paper]{amsart}
\usepackage{amsmath,amssymb,amscd,amsthm}
\usepackage{graphics}
\usepackage{enumerate}
\input xy
\xyoption{all}

\usepackage{comment}
\usepackage{tikz}
\usepackage{setspace}

\usetikzlibrary{arrows}

\usepackage[active]{srcltx}
\usepackage[notcite,notref,color]{showkeys}
\usepackage{enumerate}
 \usepackage[colorlinks,final,backref=page,hyperindex]{hyperref}

\usepackage{amsmath,amssymb,xspace,amsthm}
\newtheorem{theorem}{Theorem}[section]
\newtheorem{example}[theorem]{Example}
\newtheorem{remark}[theorem]{Remark}
\newtheorem{lemma}[theorem]{Lemma}
\newtheorem{proposition}[theorem]{Proposition}
\newtheorem{corollary}[theorem]{Corollary}
\numberwithin{equation}{section}

\newcommand{\C}{\mathbb C}

\newcommand{\Z}{\mathbb{Z}}

\def\mg{\mathfrak{g}}

\def\mm{\mathfrak{m}}
\def\mn{\mathfrak{n}}
\def\mh{\mathfrak{h}}

\def\p{\partial}

\def\sl{\mathfrak{sl}}
\def\sp{\mathfrak{sp}}
\def\gl{\mathfrak{gl}}

\def\e{\epsilon}
\def\ba{\mathbf{a}}
\def\br{\mathbf{r}}

\def\bm{\mathbf{m}}
\def\b1{\mathbf{1}}

\newcommand{\ch}{\mathcal{H}}

\title[Minimal $W$-algebra]{\bf Minimal nilpotent finite $W$-algebra and cuspidal module category of $\mathfrak{sp}_{2n}$}
\author{Genqiang Liu and Mingjie Li}

\date{\today}

\begin{document}
\begin{abstract}Let $U_S$ be the localization of $U(\mathfrak{sp}_{2n})$
with respect to the Ore subset $S$ generated by the root vectors
$X_{\epsilon_1-\epsilon_2},\dots,X_{\epsilon_1-\epsilon_n},
X_{2\epsilon_1}$.
We show that the minimal nilpotent finite $W$-algebra $W(\mathfrak{sp}_{2n}, e)$ is isomorphic to the centralizer $C_{U_S}(B)$ of some subalgebra $B$ in $U_S$, and it can be identified with a tensor product factor of $U_S$. As an application, we show that  the  category of weight $\mathfrak{sp}_{2n}$-modules
 with injective actions of all  root vectors and finite-dimensional weight spaces
 is equivalent to the category of finite-dimensional modules
over $W(\mathfrak{sp}_{2n}, e)$, explaining the coincidence that both of them are semi-simple.
\end{abstract}
\vspace{5mm}
\maketitle
\noindent{{\bf Keywords:}  Finite $W$-algebra, cuspidal module, centralizer, equivalence}
\vspace{2mm}

\noindent{{\bf Math. Subj. Class.} 2020: 17B05, 17B10, 17B20, 17B35}

\section{introduction}

The theory of  $W$-algebras shows up in the two-dimensional conformal field theory, see \cite{Z}. Finite $W$-algebras and affine $W$-algebras
are important $W$-algebras. A finite $W$-algebra $W(\mg, e)$ is certain associative algebra associated to a complex simple Lie algebra $\mg$ and a nilpotent element $e \in\mg$. Affine $W$-algebras are  vertex algebras introduced by using quantization of the Drinfeld-Sokolov reduction,  see \cite{FF,KRW}.
Finite $W$-algebras can be realized as the Zhu algebras of corresponding affine $W$-algebras \cite{DK}.
 In the literature of mathematical physics, finite
$W$-algebras appeared in the work of de Boer and Tjin \cite{BT} from the viewpoint
of BRST quantum hamiltonian reduction.   It is a generalization of the
universal enveloping algebra $U(\mg)$. In mathematics, the concept of finite $W$-algebras can be traced back to  the study of Whittaker vectors and Whittaker
modules in the famous  work of \cite{K},  see also \cite{Ly}.
The general definition of
finite $W$-algebras were introduced by Premet \cite{Pr1}. The finite $W$-algebras are quantizations of Poisson algebras of functions on the
Slodowy  slice at $e$ to the adjoint orbit of $e$. Finite $W$-algebras  and their representation theory have been extensively studied by mathematicians and physicists, see \cite{ BG, GG, BK, L1, L2, LO, Pe, PT,Pr2, T, W}.

From the definition of finite $W$-algebras, it is necessary and important to study various interactions between finite $W$-algebras and corresponding universal enveloping  algebras $U(\mg)$. When the adjoint orbit of $e$  has the minimal dimension,  the gap between $W(\mg, e)$   and $U(\mg)$ is very small, see \cite{Pr2,Pe}. In this case, $W(\mg, e)$ is called the  minimal nilpotent finite $W$-algebra.
 In the present paper, we use the tensor product decomposition
to research the minimal nilpotent finite $W$-algebra $W(\sp_{2n},e)$ for $\sp_{2n}$.
One of our main results is the following
theorem.

\begin{theorem}\label{theorem1} There is an algebra isomorphism $U_S\cong D^\p_n\otimes W(\sp_{2n},e)$, where $D_n^{\p}$ is  the localization of the  Weyl algebra $D_n$ with respect to the Ore subset
$\{\p_1^{i_1}\dots\p_n^{i_n}\mid i_1,\dots,i_n\in \Z_{\geq 0}\}$.
\end{theorem}

This tensor product decomposition makes the Skryabin equivalence (see the Appendix in \cite{Pr1}) between the category of Whittaker $\sp_{2n}$-modules and the category of  $W(\sp_{2n},e)$-modules very natural, see Remark \ref{equi-remark}.
 As an application, we give a very explicit presentation of the unique one dimensional representation of  $W(\sp_{2n},e)$, see Proposition \ref{one-dim}. One dimensional representations for any finite $W$-algebra were studied in \cite{PT,T}. A weight $\sp_{2n}$-module with finite-dimensional weight spaces is called a
 cuspidal module, if every root vector acts injectively on it.
We also
show that  the cuspidal module category of $\mathfrak{sp}_{2n}$ is equivalent to the category of finite-dimensional modules
over the minimal nilpotent finite $W$-algebra $W(\sp_{2n},e)$, see Theorem \ref{equ-cw}. Cuspidal modules are very important modules, since every simple weight module $M$ with finite-dimensional weight spaces over a reductive finite-dimensional Lie algebra $\mg$ is isomorphic to
the unique simple quotient of a parabolically induced generalized Verma module $U(\mg)\otimes _{U(\mathfrak{p})} S$,
where $S$ is a cuspidal module over the Levi component of $\mathfrak{p}$, see \cite{F,M}. A result of Fernando implies that only  simple Lie algebras of type $A$ and $C$ have cuspidal modules.
In view of
Theorem \ref{theorem1}, it will be very interesting to describe the noncommutative fractional field of $W(\sp_{2n},e)$, since this might help to resolve the Gelfand-Kirillov conjecture for $\sp_{2n}$. In the theory of vertex algebras, centralizer (called coset in vertex algebras) construction
is also one of the major ways to construct new vertex operator algebras from given ones. For example,  the coset construction
of the minimal series principal affine $W$-algebras of ADE types was
given in \cite{ACL} in full generality.

\section{Preliminaries on the minimal nilpotent finite $W$-algebra $W(\sp_{2n},e)$}

In this section, we collect  some necessary definitions and results, including $\sp_{2n}$ and the
finite $W$-algebra $W(\sp_{2n},e)$.

\subsection{Lie algebra $\sp_{2n}$}

Fix an integral number $n\geq 2$, $J_n=\left(
                                       \begin{array}{cc}
                                        0 & I_n \\
                                          -I_n& 0 \\
                                       \end{array}
                                     \right)
$.
Recall that
$$\mathfrak{sp}_{2n}:=\{A\in \gl_{2n}\mid A^tJ_n+J_nA=0\}.$$
Any matrix in $\mathfrak{sp}_{2n}$ is of the form:
  $$\left(
      \begin{array}{cc}
        X & P \\
        Q  & -X^t \\
      \end{array}
    \right),$$
  where  $X,P,Q\in M_{n\times n}(\C)$, $P=P^t,Q=Q^t$.

Let $\{e_{i,j}| i,j=1,\dots,2n\}$ be the standard basis of  $\gl_{2n}$.
Clearly $\mathfrak{sp}_{2n}$ has the following basis:
$$\aligned h_i:&=e_{ii}-e_{n+i,n+i},\ 1\leq i\leq n,\\
X_{\e_i-\e_j}:&=e_{i,j}-e_{n+j,n+i}, \ 1\leq i\neq j\leq n,\\
X_{\e_i+\e_j}:&=e_{i,n+j}+e_{j,n+i}, \ 1\leq i\leq  j\leq n,\\
X_{-\e_i-\e_j}:&=e_{n+j,i}+e_{n+i,j}, \ 1\leq i\leq j\leq n.
\endaligned$$

Let $\mh_n:=\sum_{i=1}^n \C h_i$ which is a Cartan subalgebra of
$\mathfrak{sp}_{2n}$.  The subalgebra $\mh_n\oplus (\oplus_{1\leq i \neq  j \leq n}\C X_{\e_i-\e_j})$ is isomorphic to
$\gl_n$.
For each $i\in \{1,\dots,n\}$,
define $\e_i\in \mh_n^*$ such that $\e_i(h_j)=\delta_{ij}$.
Then for any $h\in \mh$, we have
$$\aligned\
[h,X_{\e_i-\e_j}]&=(\e_i-\e_j)(h)X_{\e_i-\e_j},\\
[h,X_{\e_i+\e_j}]&=(\e_i+\e_j)(h)X_{\e_i+\e_j},\\
[h,X_{-\e_i-\e_j}]&=-(\e_i+\e_j)(h)X_{-\e_i-\e_j}.
\endaligned$$

So the root system of $\mathfrak{sp}_{2n}$ is
$$\Phi=\{\e_i-\e_j, \e_k+\e_l,-\e_k-\e_l\mid 1\leq i\neq j \leq n,
1\leq k, l\leq n\}.$$ The positive system is
$$ \Phi^+=\{\e_i-\e_j, \e_k+\e_l,\mid 1\leq i< j \leq n,
1\leq k, l\leq n\}.$$

Set
\begin{displaymath}
\mathfrak{n}_+:=\bigoplus_{\alpha\in \Phi^+}\mathfrak{\mg}_\alpha,\ \
\mathfrak{n}_-:=\bigoplus_{\alpha\in \Phi^+}\mathfrak{\mg}_{-\alpha},
\end{displaymath}where $\mg_\alpha=\{x\in \mathfrak{sp}_{2n}\mid [h,x]=\alpha(h)x,\ \forall\ h\in\mh_n\}$.
Then the decomposition
\begin{equation}\label{eq1}
\sp_{2n}=\mathfrak{n}_{-}\oplus\mathfrak{h}_n\oplus \mathfrak{n}_{+}
\end{equation}
is a {\em triangular decomposition} of $\sp_{2n}$.

%
%

For an $\sp_{2n}$-module $V$ and $\lambda\in \C^n$, denote
$$V_\lambda=\{v\in V\mid h_iv=\lambda_iv,\  \text{for all}\ i=1,\dots,n\}.$$
A nonzero element in $V_\lambda$ is called a weight vector.
An
$\sp_{2n}$-module $V$ is called a weight module if $V=\oplus_{\lambda\in \C^n} V_\lambda$. For a weight module $V$,
set  $\mathrm{supp}(V)=\{\lambda \in \C^n\mid V_\lambda\neq 0\}$.
 A weight $\sp_{2n}$-module is called a highest weight module if it is generated by a weight vector annihilated by $\mn_+$.
 A weight $\sp_{2n}$-module $V$ with finite-dimensional weight spaces is called a
cuspidal module provided that for any $\alpha\in \Phi$, $X_\alpha$  acts injectively on $V$.

Let $A_n=\C[t_1,\dots, t_n]$. The  Weyl algebra $D_{n}$
 over $A_n$  is the unital  associative algebra
over $\mathbb{C}$ generated by $t_1,\dots,t_n$,
$\partial_1,\dots,\partial_n$ subject to
the relations
$$[\partial_i, \partial_j]=[t_i,t_j]=0,\qquad [\partial_i,t_j]=\delta_{i,j},\ 1\leq i,j\leq n.$$

The following homomorphism  can be verified with a direct computation, see \cite{BL}.
\begin{proposition}\label{D-map} There is an associative algebra  homomorphism
$$\phi: U(\sp_{2n})\rightarrow D_n$$ such that
$$\aligned
&h_i\mapsto t_i\p_i+\frac{1}{2},\ \
X_{\e_i-\e_j}\mapsto t_i\p_j, i \neq j,\\
&X_{\e_i+\e_j}\mapsto t_it_j, \ \
X_{-\e_i-\e_j}\mapsto -\p_i\p_j.
\endaligned$$
\end{proposition}

Any module $M$ over $D_n$ can be defined to be an $\sp_{2n}$-module
 through the homomorphism $\phi$. For example, if $M=(\C[t_1^{\pm 1}]/\C[t_1])\otimes \cdots  \otimes(\C[t_n^{\pm 1}]/\C[t_n])$, then the $\sp_{2n}$-submodule of  $M$ generated by $t_1^{-1}\otimes \dots \otimes t_n^{-1}$ is isomorphic to the simple highest weight module $L(-\frac{\b1}{2})$, where $-\frac{\b1}{2}=(-\frac{1}{2},\dots, -\frac{1}{2})$.

\subsection{Minimal nilpotent finite $W$-algebra $W(\sp_{2n},e)$ and its twisted form}
The general theory of finite $W$-algebras were introduced in \cite{Pr1}.
Let $h=-h_1, e=e_{n+1,1}, f=e_{1,n+1}$. Then $\{e,h,f\}$
is an $\sl_2$-triple. By \cite{Pr2}, $e$ is a minimal nilpotent element of
$\sp_{2n}$.
 The
eigenspace decomposition of the adjoint action
$\text{ad} h: \sp_{2n}\rightarrow \sp_{2n}$  provides a $\Z$-grading
$$\sp_{2n}=\sp_{2n}(-2)\oplus \sp_{2n}(-1)\oplus\sp_{2n}(0)\oplus \sp_{2n}(1)\oplus \sp_{2n}(2),$$
that is, $\sp_{2n}(i)=\{x\in \sp_{2n}\mid [h,x]=ix\}$, $i=0,1,-1,2,-2$.
This $\Z$-grading of $\sp_{2n}$
will be referred to as a {\it Dynkin $\Z$-grading} which is good for $e$.
We can see that
$$\aligned
\sp_{2n}(2)&=\C X_{-2\e_1},\ \ \sp_{2n}(-2)=\C X_{2\e_1},\\
\sp_{2n}(1)&= \oplus_{i=2}^n(\C X_{\e_i-\e_1}\oplus \C X_{-\e_i-\e_1}),\\
\sp_{2n}(-1)&=\oplus_{i=2}^n(\C X_{\e_1-\e_i}\oplus \C X_{\e_i+\e_1}),\\
\sp_{2n}(0)&=
\oplus(\oplus_{2\leq i\neq j\leq n}\C X_{\e_i-\e_j})
 \oplus(\oplus_{2\leq k\leq l\leq n}\C X_{\e_k+\e_l})\\
&\ \ \  \oplus(\oplus_{2\leq k\leq l\leq n}\C X_{-\e_k-\e_l})\oplus \mh_n.
\endaligned $$

Define a skew-symmetric bilinear form $(\cdot,\cdot)$ on $\sp_{2n}(-1)$ by
$$(x,y):=\mathrm{tr}([x,y]e),\ \text{for all}\ x,y\in \sp_{2n}(-1).$$
Since $\mathrm{ad}e: \sp_{2n}(-1)\rightarrow \sp_{2n}(1)$ is bijective, the bilinear form  $(\cdot,\cdot)$  on $\sp_{2n}(-1)$ is non-degenerate. One can see that $\oplus_{i=2}^n\C X_{\e_1-\e_i}$ is a Lagrangian (i.e. maximal
isotropic) subspace of $\sp_{2n}(-1)$ with respect to the form $(\cdot,\cdot)$.
 Let $$\mm_n=(\oplus_{j=2}^n\C X_{\e_1-\e_j})\oplus\C X_{2\e_1},$$ which is a commutative
subalgebra of $\sp_{2n}$.
The map
$$\theta:\mm_n \rightarrow \C, X\mapsto \frac{1}{2}\text{tr}(Xe),$$ defines a
one-dimensional
$\mm_n $-module $\C_{\theta}:=\C v_{\b1}$. It is clear that $$\theta(X_{2\e_1})=1, \theta(X_{\e_1-\e_j})=0, \ \text{for any}\   j\neq 1.$$
Define  the induced $\sp_{2n}$-module
(called generalized Gelfand-Graev module)
$$Q_{e} := U(\sp_{2n}) \otimes_{U(\mm_n)}\C_{\theta}.$$
The finite $W$-algebra $W(\sp_{2n},e)$ is defined to be
the endomorphism algebra
$$W(\sp_{2n},e):= \text{End}_{\sp_{2n}}(Q_{e} )^{\text{op}}. $$
This is the original definition of finite $W$ algebras defined by Premet \cite{Pr1}.

It was shown by Premet that the associated graded algebra  $\mathrm{Gr}(W(\sp_{2n},e))$  of $W(\sp_{2n},e)$
with respect to the Kazhdan filtration is isomorphic to the symmetric  algebra $S(\sp_{2n}^e)$ of the centralizer $\sp_{2n}^e$ of $e$ in $\sp_{2n}$, see Theorem 4.6 in \cite{Pr1}.

By the definition of $\theta$, $X_{\e_1-\e_j}$ acts locally nilpotently
on $Q_e$, for any $2\leq j\leq n$.
In order to obtain a new $\sp_{2n}$-module  from $Q_e$ with injective action of $X_{\e_1-\e_j}$, we will twist $Q_e$ by a suitable automorphism of $\sp_{2n}$. This idea is motivated by the work of Kostant  \cite{K}. When a Lie homomorphism $\eta: \mn_+\rightarrow \C$ satisfies that $\eta(X_{\alpha})\neq 0$ for any simple root $\alpha$, Kostant found that Whittaker vectors in the universal Whittaker module $Y_\eta$ are closely related with the center of $U(\sp_{2n})$.

Let $X=X_{\e_2+\e_1}+\dots+X_{\e_n+\e_1}$ and $\sigma$ be the inner automorphism of $\sp_{2n}$ defined the inner derivation $-\mathrm{ad} X$.
From $$[e_{1,i}-e_{n+i,n+1}, e_{1,n+j}+e_{j,n+1}]=2\delta_{ij}e_{1,n+1}=\delta_{ij}X_{2\e_1}, \ i,j \neq 1,$$
it can be checked that
$\sigma(X_{\e_1-\e_i})=X_{\e_1-\e_i}+X_{2\e_1}, \sigma(X_{2\e_1})=X_{2\e_1}$, for any  $i\neq 1$.

The $\sp_{2n}$-module $Q_e$ can be twisted by $\sigma$ to be a new $\sp_{2n}$-module $Q_e^{\sigma}$.   Explicitly the action of
$\sp_{2n}$ on  $Q_e^{\sigma}$ is defined by
\begin{equation}\label{twisted-mod}u\cdot v=\sigma(u)v, \ \ u\in \sp_{2n}, v\in Q_e.\end{equation}
So in $Q_e^{\sigma}$, we have that $X_{\e_1-\e_i}\cdot v_\b1=v_\b1, X_{2\e_1}\cdot v_\b1=v_\b1$ for any $i>1$. Define $$W^\sigma(\sp_{2n},e):= \text{End}_{\sp_{2n}}(Q^\sigma_{e} )^{\text{op}}. $$

\begin{lemma}$W(\sp_{2n},e)\cong W^\sigma(\sp_{2n},e)$.
\end{lemma}
\begin{proof}
From $$g(u\cdot v)=g(\sigma(u)v)
=\sigma(u)g(v)=u\cdot g(v),\ \forall \  u\in \sp_{2n}, \ v\in Q_e,$$ it follows that
the identity map $$W(\sp_{2n},e)\rightarrow W^\sigma(\sp_{2n},e): g\mapsto g $$ is a well-defined algebra isomorphism.
\end{proof}

\begin{remark} In fact, we also have that $W^\sigma(\sp_{2n},e) \cong W(\sp_{2n},\sigma^{-1}(e))$, since  $[\sigma(u), e]=0$ if only if $[u, \sigma^{-1}(e)]=0$ for all $u\in \sp_{2n}$.
\end{remark}

We denote the vector $(1, 1,\dots, 1)\in \C^n$ by $\b1$.
Let $\widetilde{\ch}_{\b1}$ be the  category
whose objects are $\sp_{2n}$-modules $M$ satisfying that  $X_{2\e_1}-1$,
 $X_{\e_1-\e_2}-1,\dots, X_{\e_1-\e_n}-1$ act locally nilpotently on $M$.
 For an $\sp_{2n}$-module $M$ in $\widetilde{\ch}_{\b1}$,
denote
 $$\mathrm{wh}_{\b1}(M)
=\{ v\in M\mid X_{\e_1-\e_2}v=\dots=X_{\e_1-\e_n}v=X_{2\e_1}v=v\}.$$
Any nonzero vector in $\mathrm{wh}_{\b1}(M)$ is called a Whittaker vector. A module in $\widetilde{\ch}_{\b1}$ is called a Whittaker module.

 By the equivalence given by Skryabin (see the Appendix in \cite{Pr1}) and isomorphism between $W^\sigma(\sp_{2n},e)$ and $W(\sp_{2n},e)$, we have the following equivalence.
 \begin{lemma}
 The functors $$M\mapsto \mathrm{wh}_{\b1}(M),
\ \ \ \  V\mapsto Q^\sigma_{e}\otimes_{W^\sigma(e)} V, $$ are inverse equivalence
 between $\widetilde{\ch}_{\b1}$ and the
category of $W^\sigma(\sp_{2n}, e)$-modules.
\end{lemma}

Let $\mathcal{H}_{\b1}$ be the full subcategory of $\widetilde{\mathcal{H}}_{\b1}$ consisting of
modules $M$ such that $\mathrm{wh}_{\b1}(M)$ is finite-dimensional.
From the isomorphism $W^\sigma(\sp_{2n}, e)\cong W(\sp_{2n}, e)$, $\mathcal{H}_{\b1}$ is equivalent to the category
$W$-$\mathrm{fmod}$ of finite-dimensional
modules over  $W(\sp_{2n}, e)$.  One of our goals is  to find the connection between $W$-$\mathrm{fmod}$ and the category of weight modules over $\sp_{2n}$

\section{Centralizer realization of the finite $W$-algebra $W(\sp_{2n},e)$}
In this section, we  give generators of the finite $W$-algebra $W(\sp_{2n},e)$ in certain  localization of $U(\sp_{2n})$, and its explicit realization in terms of centralizer.

\subsection{Centralizer of $\mh_n\oplus \mm_n$}
Since $\text{ad} X_{2\e_1}, \text{ad}X_{\e_1-\e_2},\dots, \text{ad}X_{\e_1-\e_n}$ act locally nilpotent on $U(\sp_{2n})$, the set $S:=\{X_{2\e_1}^{i_1}X_{\e_1-\e_2}^{i_2}\dots X_{\e_1-\e_n}^{i_n}\mid i_1,\dots, i_n\in\Z_{\geq0}\}$ is an Ore subset of $U(\sp_{2n})$, see Lemma 4.2 in \cite{M}.
   Denote
 by $U_{S}$  the localization $U(\sp_{2n})$ with respect to $S$.
The reason of introducing $U_{S}$ is that
$X_{2\e_1},X_{\e_1-\e_2},\dots, X_{\e_1-\e_n}$ act bijectively on the module $Q_e^{\sigma}$.

In order to find distinguished  Whittaker vectors in $Q^\sigma_e$,  we define the following elements in $U_S$:
 \begin{equation}\label{x-w-def}
 \aligned
  A_{\e_k+\e_1}&=X_{\e_k+\e_1}X_{\e_1-\e_k}X_{2\e_1}^{-1}
-h_k,\\
A_{\e_k+\e_l}&=X_{\e_k+\e_l}X_{\e_1-\e_k}X_{\e_1-\e_l}X_{2\e_1}^{-1}
-X_{\e_1+\e_l}h_kX_{\e_1-\e_l}X_{2\e_1}^{-1}\\
&\ \ \  -X_{\e_1+\e_k}h_lX_{\e_1-\e_k}X_{2\e_1}^{-1}
+h_l(h_k-\delta_{lk}),\\
  A_{\e_i-\e_j}&=X_{\e_i-\e_j}X_{\e_1-\e_i}X_{\e_1-\e_j}^{-1}-h_i,\ \   i\neq j, \\
A_{\e_k-\e_1}&=X_{\e_k-\e_1}X_{\e_1-\e_k}
 +\sum_{j=2, j\neq k}^n X_{\e_k-\e_j}(h_j-1)X_{\e_1-\e_k}X_{\e_1-\e_j}^{-1}\\
 & \ \ + h_k(h_k-1)-X_{\e_k+\e_1}(I_n+2)X_{\e_1-\e_k}X_{2\e_1}^{-1},
 \endaligned\end{equation}
 and
  \begin{equation}\label{x-w-def1}
 \aligned
A_{-2\e_1}
&= X_{-2\e_1}X_{2\e_1}
+\sum_{k=2}^n2X_{-\e_1-\e_k}h_kX_{\e_1-\e_k}^{-1}X_{2\e_1}
-\sum_{k=2}^nX_{-2\e_k}h_kX_{\e_1-\e_k}^{-2}X_{2\e_1}\\
&\ \ \ +\sum_{k,l=2}^n X_{-\e_k-\e_l}h_kh_lX_{\e_1-\e_k}^{-1}X_{\e_1-\e_l}^{-1}X_{2\e_1}
+(I_n+2n)I_n,\\
A_{-\e_1-\e_k}
&= X_{-\e_1-\e_k}X_{\e_1-\e_k}^{-1}X_{2\e_1}
+\sum_{l=2}^n X_{-\e_k-\e_l}h_lX_{\e_1-\e_k}^{-1}X_{\e_1-\e_l}^{-1}X_{2\e_1}
+I_n,\\
 A_{-\e_k-\e_l}
&=X_{-\e_k-\e_l}X_{\e_1-\e_k}^{-1}X_{\e_1-\e_l}^{-1}X_{2\e_1},
\endaligned\end{equation}
where $I_n=h_1+\cdots+h_n$, $2 \leq i,j, k, l\leq n$.

Let $B$  be the subalgebra
 of $U_S$ generated by
 $$h_1,\dots,h_n, X_{\e_1-\e_2}^{\pm 1},\dots, X_{\e_1-\e_n}^{\pm 1}, X_{2\e_1}^{\pm 1},$$
 and $C_{U_S}(B)$ be the commutant  of $B$ in $U_S$, i.e.,
 $C_{U_S}(B)=\{u\in U_S\mid [u, B]=0\}$.

By direct computations, we have the following lemma.

\begin{lemma} \label{comm-lemm}For any $i,j,k,l\in\{2,\dots,n\}$ with $i\neq j$, we have that
 $$A_{\e_k+\e_l}, A_{\e_k+\e_1},
 A_{\e_i-\e_j}, A_{\e_k-\e_1}, A_{-2\e_1},
A_{-\e_1-\e_k},
A_{-\e_l-\e_k}\in C_{U_S}(B).$$
\end{lemma}

\begin{proof}
By the expressions in (\ref{x-w-def}) and (\ref{x-w-def1}), all considered elements commute with
$\mh_n$. We will verify  that these elements commute with $\mm_n$ one by one. Note that  $\mm_n=(\oplus_{j=2}^n\C X_{\e_1-\e_j})\oplus\C X_{2\e_1}$.

(1) From $$\aligned\  & [X_{\e_1-\e_q},X_{\e_k+\e_l}]
 &=\delta_{qk}X_{\e_1+\e_l}+\delta_{ql}X_{\e_1+\e_k},
\endaligned$$
 we can see that
 $$\aligned\  [X_{\e_1-\e_q},A_{\e_k+\e_l}]&=
[X_{\e_1-\e_q},X_{\e_k+\e_l}X_{\e_1-\e_k}X_{\e_1-\e_l}X_{2\e_1}^{-1}
-X_{\e_1+\e_l}h_kX_{\e_1-\e_l}X_{2\e_1}^{-1}\\
&\ \ \  -X_{\e_1+\e_k}h_lX_{\e_1-\e_k}X_{2\e_1}^{-1}
+h_l(h_k-\delta_{lk})]\\
&= (\delta_{qk}X_{\e_1+\e_l}+\delta_{ql}X_{\e_1+\e_k})
X_{\e_1-\e_k}X_{\e_1-\e_l}X_{2\e_1}^{-1}\\
&\ \ \  -\delta_{ql}X_{2\e_1}h_kX_{\e_1-\e_l}X_{2\e_1}^{-1}- \delta_{qk}X_{\e_1+\e_l}X_{\e_1-\e_k}X_{\e_1-\e_l}X_{2\e_1}^{-1}\\
&\ \ \  - \delta_{qk}X_{2\e_1}h_lX_{\e_1-\e_k}X_{2\e_1}^{-1}- \delta_{ql}X_{\e_1+\e_k}X_{\e_1-\e_l}X_{\e_1-\e_k}X_{2\e_1}^{-1}\\
&\ \ \ + \delta_{ql}h_kX_{\e_1-\e_l}
 + \delta_{qk}h_lX_{\e_1-\e_k}\\
&= -\delta_{ql}X_{2\e_1}h_kX_{\e_1-\e_l}X_{2\e_1}^{-1}
  -\delta_{qk}X_{2\e_1}h_lX_{\e_1-\e_k}X_{2\e_1}^{-1}\\
 & \ \ \  + \delta_{ql}h_kX_{\e_1-\e_l}
 + \delta_{qk}h_lX_{\e_1-\e_k}\\
 &= 0.
\endaligned$$
Then by $ [X_{2\e_1},A_{\e_k+\e_l}]=0$, we have that $A_{\e_k+\e_l}\in C_{U_S}(B)$.

(2) It can be verified that
$$\aligned\  [X_{\e_1-\e_q},A_{\e_k+\e_1}]&=
[X_{\e_1-\e_q},X_{\e_k+\e_1}X_{\e_1-\e_k}X_{2\e_1}^{-1}
-h_k]\\
&= \delta_{qk}X_{2\e_1}X_{\e_1-\e_k}X_{2\e_1}^{-1}-\delta_{qk}X_{\e_1-\e_k}\\
 &=0,
\endaligned$$
and
$$\aligned\  [X_{2\e_1},A_{\e_k+\e_1}]&=
[X_{2\e_1},X_{\e_k+\e_1}X_{\e_1-\e_k}X_{2\e_1}^{-1}
-h_k]=0,
\endaligned$$
since $ k\neq 1$.

(3) Since $i,j\in\{2,\dots,n\}$,  $[X_{2\e_1},A_{\e_i-\e_j}]=0$. For any $k\in\{2,\dots,n\}$, we have that
$$\aligned\  [X_{\e_1-\e_k},A_{\e_i-\e_j}]&=[X_{\e_1-\e_k}, X_{\e_i-\e_j}X_{\e_1-\e_i}X_{\e_1-\e_j}^{-1}-h_i]\\
&=\delta_{ki}X_{\e_1-\e_j}X_{\e_1-\e_i}X_{\e_1-\e_j}^{-1}-\delta_{ki}X_{\e_1-\e_i}\\
&= 0.
\endaligned$$

(4)
From $$\aligned
\ [X_{\e_k-\e_1}, X_{2\e_1}]&=[e_{k1}-e_{n+1,n+k},2e_{1,n+1}]\\
&=2e_{k,n+1}+2e_{1,n+k}=2X_{\e_1+\e_k},
\endaligned$$
it can be computed that
$$\aligned\  & [X_{2\e_1},A_{\e_k-\e_1}]\\
&=[X_{2\e_1}, X_{\e_k-\e_1}X_{\e_1-\e_k}
 +\sum_{j=2}^n X_{\e_k-\e_j}(h_j-1)X_{\e_1-\e_k}X_{\e_1-\e_j}^{-1}\\
 & \ \  -X_{\e_k+\e_1}X_{\e_1-\e_k}(I_n+2)X_{2\e_1}^{-1}]\\
 &=-2X_{\e_1+\e_k}X_{\e_1-\e_k}+2X_{\e_1+\e_k}X_{\e_1-\e_k}\\
 &=0.
\endaligned$$

 From $$\aligned\  [X_{\e_k+\e_1}, X_{\e_1-\e_q}]
 =
 -\delta_{qk}X_{2\e_1},
 \endaligned$$

it follows that
$$\aligned\  [X_{\e_1-\e_q},A_{\e_k-\e_1}]&=[X_{\e_1-\e_q},
 X_{\e_k-\e_1}X_{\e_1-\e_k}
 +\sum_{j=2}^n X_{\e_k-\e_j}(h_j-1)X_{\e_1-\e_k}X_{\e_1-\e_j}^{-1}\\
 & \ \  -X_{\e_k+\e_1}(I_n+2)X_{\e_1-\e_k}X_{2\e_1}^{-1}]\\
 &=\delta_{qk}h_1X_{\e_1-\e_k}-X_{\e_k-\e_q}X_{\e_1-\e_k}
 +\sum_{j=2}^n \delta_{qk} X_{\e_1-\e_j}(h_j-1)X_{\e_1-\e_k}X_{\e_1-\e_j}^{-1}\\
& \ \ \ +X_{\e_k-\e_q}X_{\e_1-\e_q}X_{\e_1-\e_k}X_{\e_1-\e_q}^{-1}
  -\delta_{qk}X_{2\e_1}(I_n+2)X_{\e_1-\e_k}X_{2\e_1}^{-1} \\
  &=\delta_{qk}h_1X_{\e_1-\e_k}
 +\sum_{j=2}^n \delta_{qk} h_jX_{\e_1-\e_k}
  -\delta_{qk}I_nX_{\e_1-\e_k}\\
 &=0.
\endaligned$$

(5) From
$$\aligned\   [X_{2\e_1},X_{-\e_1-\e_k}] &
=[2 e_{1,n+1}, e_{n+1,k}+e_{n+k,1}]\\
&=2e_{1,k}-2e_{n+k,n+1}=2 X_{\e_1-\e_k},
\endaligned$$

we obtain that
$$\aligned\   [X_{2\e_1}, A_{-2\e_1}]
&=[X_{2\e_1}, X_{-2\e_1}X_{2\e_1}
+\sum_{k=2}^n2X_{-\e_1-\e_k}h_kX_{\e_1-\e_k}^{-1}X_{2\e_1}
-\sum_{k=2}^nX_{-2\e_k}h_kX_{\e_1-\e_k}^{-2}X_{2\e_1}\\
&\ \ \ +\sum_{k,l=2}^n X_{-\e_k-\e_l}h_kh_lX_{\e_1-\e_k}^{-1}X_{\e_1-\e_l}^{-1}X_{2\e_1}
+(I_n+2n)I_n
]\\
&=4h_1X_{2\e_1}+\sum_{k=2}^n4X_{\e_1-\e_k}h_kX_{\e_1-\e_k}^{-1}X_{2\e_1}
-2X_{2\e_1}I_n-2(I_n+2n)X_{2\e_1}\\
&=4h_1X_{2\e_1}+\sum_{k=2}^n4(h_k+1)X_{2\e_1}-4 (I_n+n-1)X_{2\e_1}\\
&=0.
\endaligned$$

From
$$\aligned\   [X_{\e_1-\e_q}, X_{-2\e_1}]&=[e_{1,q}-e_{n+q,n+1}
, 2e_{n+1,1}]\\
&=-2e_{n+1,q}-2e_{n+q,1}=-2X_{-\e_1-\e_q},
\endaligned$$
and
$$\aligned\   [X_{\e_1-\e_q}, X_{-\e_1-\e_k}]&
=[e_{1,q}-e_{n+q,n+1}, e_{n+1,k}+e_{n+k,1}]\\
&= -e_{n+k,q}-e_{n+q,k}=-X_{-\e_k-\e_q},
\endaligned$$
we have that
$$\aligned\   [X_{\e_1-\e_q}, A_{-2\e_1}]
&=[X_{\e_1-\e_q}, X_{-2\e_1}X_{2\e_1}
+\sum_{k=2}^n2X_{-\e_1-\e_k}h_kX_{\e_1-\e_k}^{-1}X_{2\e_1}
-\sum_{k=2}^nX_{-2\e_k}h_kX_{\e_1-\e_k}^{-2}X_{2\e_1}\\
&\ \ \ +\sum_{k,l=2}^n X_{-\e_k-\e_l}h_kh_lX_{\e_1-\e_k}^{-1}X_{\e_1-\e_l}^{-1}X_{2\e_1}
+(I_n+2n)I_n,]\\
&=-2X_{-\e_1-\e_q}X_{2\e_1}
-2\sum_{k=2}^nX_{-\e_k-\e_q}h_kX_{\e_1-\e_k}^{-1}X_{2\e_1}
+2X_{-\e_1-\e_q}X_{2\e_1}-X_{-2\e_q}X_{\e_1-\e_q}^{-1}X_{2\e_1}\\
& \ \ \ +\sum_{k=2}^nX_{-\e_k-\e_q}h_kX_{\e_1-\e_k}^{-1}X_{2\e_1}
+\sum_{l=2}^nX_{-\e_l-\e_q}(h_l+\delta_{ql})X_{\e_1-\e_l}^{-1}X_{2\e_1}
\\
&= 0.
\endaligned$$

(6)
From $$\aligned\  [X_{\e_1-\e_q}, X_{-\e_k-\e_l}]
&= [e_{1,q}-e_{n+q,n+1}, e_{n+k,l}+e_{n+l,k}]=0, k,l\neq 1,
\endaligned$$
we get that
 $$\aligned\
[X_{\e_1-\e_q}, A_{-\e_1-\e_k}]
&=[X_{\e_1-\e_q}, X_{-\e_1-\e_k}X_{\e_1-\e_k}^{-1}X_{2\e_1}
+\sum_{l=2}^n X_{-\e_k-\e_l}h_lX_{\e_1-\e_k}^{-1}X_{\e_1-\e_l}^{-1}X_{2\e_1}
+I_n]\\
&= -X_{-\e_k-\e_q}X_{\e_1-\e_k}^{-1}X_{2\e_1}
+X_{-\e_k-\e_q}X_{\e_1-\e_k}^{-1}X_{2\e_1}\\
&=0,
\endaligned$$
and
 $$\aligned\
[X_{2\e_1}, A_{-\e_1-\e_k}]
&=[X_{2\e_1}, X_{-\e_1-\e_k}X_{\e_1-\e_k}^{-1}X_{2\e_1}
+\sum_{l=2}^n X_{-\e_k-\e_l}h_lX_{\e_1-\e_k}^{-1}X_{\e_1-\e_l}^{-1}X_{2\e_1}
+I_n]\\
&= 2 X_{\e_1-\e_k}X_{\e_1-\e_k}^{-1}X_{2\e_1}-2X_{2\e_1}\\
&=0.
\endaligned$$

(7)
From $$\aligned\  [X_{\e_1-\e_q}, X_{-\e_k-\e_l}]
&= [X_{2\e_1}, X_{-\e_k-\e_l}]=0, k,l\neq 1,
\endaligned$$
we obtain that
$$\aligned\
[X_{\e_1-\e_q}, A_{-\e_k-\e_l}]
&=[X_{\e_1-\e_q}, X_{-\e_k-\e_l}X_{\e_1-\e_k}^{-1}X_{\e_1-\e_l}^{-1}X_{2\e_1}]=0,
\endaligned$$
and
 $$\aligned\
[X_{2\e_1}, A_{-\e_k-\e_l}]
&=[X_{2\e_1}, X_{-\e_k-\e_l}X_{\e_1-\e_k}^{-1}X_{\e_1-\e_l}^{-1}X_{2\e_1}]=0.
\endaligned$$

\end{proof}

The next theorem tells us that $C_{U_S}(B)$ is a tensor product factor of the localized enveloping algebra
$U_{S}$.
 \begin{theorem}\label{main-iso-th}  We have that $U_{S}=B  C_{U_S}(B)\cong B\otimes  C_{U_S}(B)$ and $Z(U_{S})\cong Z(C_{U_S}(B))$.
\end{theorem}

\begin{proof}
From (\ref{x-w-def}), it follows that
\begin{equation}\label{x-def2}
 \aligned
 X_{\e_k+\e_1}
 &=X_{\e_1-\e_k}^{-1}X_{2\e_1}A_{\e_k+\e_1}
 +h_kX_{\e_1-\e_k}^{-1}X_{2\e_1},\\
 X_{\e_k+\e_l}&=X_{\e_1-\e_k}^{-1}X_{\e_1-\e_l}^{-1}X_{2\e_1}A_{\e_k+\e_l}
+X_{\e_1+\e_l}h_kX_{\e_1-\e_k}^{-1}\\
& \ \ \ +X_{\e_1+\e_k}h_lX_{\e_1-\e_l}^{-1}
-h_l(h_k-\delta_{kl})X_{\e_1-\e_k}^{-1}X_{\e_1-\e_l}^{-1}X_{2\e_1},\\
 X_{\e_i-\e_j}&=X_{\e_1-\e_i}^{-1}X_{\e_1-\e_j}A_{\e_i-\e_j}
 +h_iX_{\e_1-\e_i}^{-1}X_{\e_1-\e_j},\ \   i\neq j, \\
X_{\e_k-\e_1}&=A_{\e_k-\e_1}X_{\e_1-\e_k}^{-1}-
 \sum_{j=2, j\neq k }^n X_{\e_k-\e_j}(h_j-1)X_{\e_1-\e_j}^{-1}\\
 & \ \  -h_k(h_k-1)X_{\e_1-\e_k}^{-1}+X_{\e_k+\e_1}(I_n+2)X_{2\e_1}^{-1},
 \endaligned\end{equation}
 and
 \begin{equation}\label{x-def3}
 \aligned
 X_{-\e_k-\e_l}
&=A_{-\e_k-\e_l}X_{\e_1-\e_k}X_{\e_1-\e_l}X_{2\e_1}^{-1},\\
 X_{-\e_1-\e_k}&=
X_{\e_1-\e_k}X_{2\e_1}^{-1}A_{-\e_1-\e_k}
-\sum_{l=2}^n X_{-\e_k-\e_l}h_lX_{\e_1-\e_l}^{-1}
-I_nX_{\e_1-\e_k}X_{2\e_1}^{-1},\\
X_{-2\e_1}
&= A_{-2\e_1}X_{2\e_1}^{-1}
-\sum_{k=2}^n2X_{-\e_1-\e_k}h_kX_{\e_1-\e_k}^{-1}+\sum_{k=2}^nX_{-2\e_k}h_kX_{\e_1-\e_k}^{-2}\\
&\ \ \ -\sum_{k,l=2}^n X_{-\e_k-\e_l}h_kh_lX_{\e_1-\e_k}^{-1}X_{\e_1-\e_l}^{-1}
-(I_n+2n)I_nX_{2\e_1}^{-1}.
\endaligned\end{equation}
Observing these formulas, we see that $U_S\subseteq BC_{U_S}(B)$. Hence
$U_S=BC_{U_S}(B)$.

Since $B$ and $C_{U_S}(B)$ are subalgebras of $U_S$, we have a
natural homomorphism $\gamma$ from $B\otimes C_{U_S}(B)$ to $U_S$.
By (\ref{x-def2}) and (\ref{x-def3}), $\gamma$ is surjective. Observe that
$\gamma$ maps a monomial on $$A_{\e_k+\e_1},
A_{\e_k+\e_l}, A_{\e_i-\e_j}, A_{\e_k-\e_1}, A_{-2\e_1},
A_{-\e_1-\e_k},
A_{-\e_l-\e_k}, h_i, X_{\e_1-\e_i}, X_{2\e_1},$$
to a polynomial on $$X_{\e_k+\e_1},
X_{\e_k+\e_l}, X_{\e_i-\e_j}, X_{\e_k-\e_1}, X_{-2\e_1},
X_{-\e_1-\e_k},
X_{-\e_l-\e_k}, h_i, X_{\e_1-\e_i}, X_{2\e_1},$$
whose highest degree term is nonzero. So   $\gamma$ is also injective, and hence it is bijective. Since the center of $B$ is trivial, $Z(U_{S})\cong Z(C_{U_S}(B))$.
\end{proof}

By Theorem \ref{main-iso-th} and the PBW Theorem on $U_S$,  we have the following result.

\begin{corollary}\label{PBW-TH}
  All the ordered monomials in $$A_{\e_k+\e_1}, A_{\e_k+\e_l}, A_{\e_i-\e_j} (i\neq j),
A_{\e_k-\e_1},
A_{-2\e_1}, A_{-\e_1-\e_k}, A_{-\e_l-\e_k},$$ with $i, j,k,l=2,\dots, n $ form a basis of $C_{U_S}(B)$
over $\C$.
\end{corollary}

The PBW type Theorem for any finite $W$-algebra was given in
\cite{Pr1}.

\subsection{Isomorphism between $W(\sp_{2n}, e)$ and $C_{U_S}(B)$}

The following theorem says that all
Whittaker vectors in the module $Q^\sigma_{e}$ can be  given by using elements in $C_{U_S}(B)$.
This result is motivated by  the classical Kostant's Theorem, see \cite{K}.

\begin{theorem} \label{W-iso}We have the isomorphism $W(\sp_{2n}, e)\cong C_{U_S}(B)$.
\end{theorem}

\begin{proof}Since $W^\sigma(\sp_{2n}, e)\cong W(\sp_{2n}, e)$,
it suffices to show that $W^\sigma(\sp_{2n}, e)\cong C_{U_S}(B)$.
Since $X_{\e_1-\e_2}-1, \dots,X_{\e_1-\e_n}-1, X_{2\e_1}-1$ act locally nilpotently on $Q^\sigma_{e}$,  $Q^\sigma_{e}$ can be extended to
a $U_S$-module.
Recall that $W^\sigma(\sp_{2n},e)= \text{End}_{\sp_{2n}}(Q^\sigma_{e} )^{\text{op}}$. Since $Q^\sigma_{e}$ is generated by
$v_\mathbf{1}$, any $\eta\in W^\sigma(\sp_{2n},e)$ is determined by $\eta(v_{\b1})$.
Moreover $\eta(v_{\b1})\in \mathrm{wh}_{\b1}(Q^\sigma_{e})$. By the definition of $C_{U_S}(B)$, any
$u\in C_{U_S}(B)$ can define a $\eta_u\in W^\sigma(\sp_{2n},e)$ such that $\eta_u(v_{\b1})=u\cdot v_{\b1}$. That is, there is a homomorphism
$$\Phi: C_{U_S}(B)\rightarrow \text{End}_{\sp_{2n}}(Q^\sigma_{e} )^{\text{op}}, u\mapsto \eta_{u}.$$
Note that as a vector space $Q^\sigma_{e}\cong U(\mh_n)\otimes C_{U_S}(B)$.
By the definition of $Q^\sigma_{e}$ and decomposition
$U_S\cong C_{U_S}(B)\otimes B$, it follows that $\Phi$ is injective.

 Conversely for any  $\eta\in W^\sigma(\sp_{2n},e)$, since $Q^\sigma_{e}=U_S \cdot v_\b1$,
we can suppose that
$$\eta(v_1)=\sum_{0\leq \br\leq \bm}h^{\br}u_\br \cdot v_{\b1},$$where
$u_\br\in C_{U_S}(B), h^{\br}=h_1^{r_1}\dots h_n^{r_n}$, and $``\leq"$ is the usual
lexicographical order.
For any $i\in \{2,\dots, n\}$, since $[h_1,X_{2\e_1}]=2X_{2\e_1}, [h_i,X_{\e_1-\e_i}]=-X_{\e_1-\e_i}$, this yields
$$X_{2\e_1}h_1^{m_1}=(h_1-2)^{m_1}X_{2\e_1},\ \  X_{\e_1-\e_i}h_i^{m_i}=(h_i+1)^{m_i}X_{\e_1-\e_i},$$
and $$[X_{2\e_1},h_1^{m_1}]=\sum_{j=1}^{m_1}\binom{m_1}{j}(-2)^jh_1^{m_1-j}X_{2\e_1},\ \  [X_{\e_1-\e_i},h_i^{m_i}]
=\sum_{j=1}^{m_1}\binom{m_1}{j}h_i^{m_i-j}X_{\e_1-\e_i}.$$
Then by induction on $m_1$ and $m_i$,  we can find nonzero
 $k_{m_1}, k_{m_i}\in \C$ such that
$$(X_{2\e_1}-1)^{m_1}h_1^{m_1}v_{\b1}=k_{m_1}v_{\b1},  \ \ (X_{2\e_1}-1)^{s_1}h_1^{m_1}v_{\b1}=0, \forall\ s_1>m_1,$$
and  $$(X_{\e_1-\e_i}-1)^{m_i}h_i^{m_i}v_{\b1}=k_{m_i}v_{\b1},  \ \ (X_{\e_1-\e_i}-1)^{s_i}h_i^{m_i}v_{\b1}=0, \forall\ s_i>m_i.$$
Consequently
$$\aligned  u_{\bm}v_{\b1}=&\frac{1}{k_{m_1}\dots k_{m_n}} (X_{\e_1-\e_n}-1)^{m_n}\dots (X_{\e_1-\e_2}-1)^{m_2}(X_{2\e_1}-1)^{m_1}\eta(v_1)=0.\endaligned$$
Similarly we have that $u_{\br}v_{\b1}=0$ for any $0< \br\leq \bm$.
So $\eta(v_1)=u_0\cdot v_{\b1}$ for some $u_0\in C_{U_S}(B)$. Therefore
$\Phi$ is also surjective, and hence it is isomorphic.
\end{proof}

Let $D_n^{\p}$ be the localization of $D_n$ with respect to the Ore subset
$$\{\p_1^{i_1}\dots\p_n^{i_n}\mid i_1,\dots,i_n\in \Z_{\geq 0}\}.$$ We can verify the following map is an algebra isomorphism:
$$\aligned D_n^{\p}&\rightarrow B,\\
 \p_1&\mapsto X_{2\e_1},\\
 t_1&\mapsto -\frac{1}{2}(h_1+h_2+\dots+h_n )X_{2\e_1}^{-1},\\
 \p_i&\mapsto X_{\e_1-\e_i},\\
 t_i&\mapsto h_i X_{\e_1-\e_i}^{-1}, i=2,\dots,n.\\
  \endaligned$$

By Theorems \ref{main-iso-th} and \ref{W-iso}, we have the following
tensor product decomposition.

\begin{theorem}\label{refine} There is an algebra isomorphism $U_S\cong D^\p_n\otimes W(\sp_{2n},e)$.
\end{theorem}

\begin{remark}
In \cite{Pr2}, Premet showed that the localization $U_f$ of $U(\sp_{2n})$ with respect to a long root vector
is isomorphic to a subalgebra of $D^\p_n\otimes W(\sp_{2n},e)$.
Our result can be viewed as a refinement of Premet's result.
On minimal nilpotent finite $W$-algebras of type $A_n$, it was also shown that a suitable localization of
$U(\sl_{n+1})$ is isomorphic to $D^\p_n\otimes W(\sl_{n+1},e)$ in \cite{LL}.
 It would be an interesting problem to extend these tensor product decompositions to other type minimal nilpotent finite $W$-algebras.
\end{remark}

\begin{remark}\label{equi-remark} By Theorem \ref{W-iso}, we can identify $C_{U_S}(B)$ with $W(\sp_{2n},e)$. Theorems \ref{main-iso-th}, \ref{W-iso} and \ref{refine} make the equivalence between $\widetilde{\ch}_{\b1}$ and the category of  $W(\sp_{2n},e)$-modules very natural. Since $[C_{U_S}(B), \mm_n]=0$,
 $\mathrm{wh}_{\b1}(M)$ is a $C_{U_S}(B)$-module for any $M\in \widetilde{\ch}_{\b1}$. Conversely for a
 $C_{U_S}(B)$-module $V$, let
 $$X_{\e_1-\e_2}v=\dots=X_{\e_1-\e_n}v=X_{2\e_1}v=v,$$ for any
 $v\in V$, consider the induced module $U_S\otimes_{U(\mm_n)_S C_{U_S}(B)} V$. Then
 the functors $$M\mapsto \mathrm{wh}_{\b1}(M),
\ \ \ \  V\mapsto U_S\otimes _{U(\mm_n)_S C_{U_S}(B)} V, $$ are inverse equivalence
 between $\widetilde{\ch}_{\b1}$ and the
category of $C_{U_S}(B)$-modules.
\end{remark}

\subsection{One dimensional representations of $W(\sp_{2n},e)$}

By Corollary 4.1 in \cite{Pr2}, $C_{U_S}(B)$ has only one  ideal of
codimension $1$, i.e.,  $C_{U_S}(B)$ has unique one dimensional module. From the following example of a Whittaker module over $\sp_{2n}$,
we can find this one dimensional module.

\begin{example}\label{whittaker-example} Let $\tau$ be the automorphism of $D_n$ such that
$$ t_i\mapsto t_i+\delta_{i1},\ \  \p_i\mapsto \p_i+1- \delta_{i1}, \  i=1,\dots, n.$$
The $D_n$-module $P=:(\C[t_1^{\pm 1}]/ \C[t_1])\otimes \C[t_2,\dots,t_n]$
can be twisted by $\tau$ to be a new $D_n$-module $P^{\tau}$. Under the map $\phi$ in Propostion \ref{D-map}, $P^{\tau}$ becomes an $\sp_{2n}$-module.
 The image of $t_1^{-1}$ in $P$ is also denoted by
 $t_1^{-1}$. Note that
 $$t_1\cdot t_1^{-1}=\p_i\cdot t_1^{-1}=t_1^{-1}, \ \text{for all}\  2\leq i \leq n.$$
Then it can be verified that the $\sp_{2n}$-submodule $P_1^\tau$ of $P^\tau$
generated by $t_1^{-1}$ is an irreducible module in $\widetilde{\ch}_{\b1}$ with $\dim\mathrm{wh}_{\b1}(P_1^\tau)=1$.
\end{example}

The following result gives an explicit characterization of the unique one dimensional $C_{U_S}(B)$-module using the  generators in Lemma \ref{comm-lemm}.
\begin{proposition} \label{one-dim}The unique  one dimensional $C_{U_S}(B)$-module $V_{-\frac{1}{2}}$ is
defined by the map:
$$\aligned
&A_{\e_k+\e_1}\mapsto -\frac{1}{2}, \ \  A_{\e_k+\e_l}\mapsto \frac{1}{4}-\frac{\delta_{lk}}{2},\ \  A_{-\e_k-\e_l}\mapsto -1,\\
&A_{\e_i-\e_j}\mapsto- \frac{1}{2}, \ \  A_{\e_k-\e_1}\mapsto -\frac{1}{4},
\ \  A_{-\e_k-\e_1}\mapsto -\frac{1}{2}-n,
\\
&A_{-2\e_1}\mapsto -n^2+\frac{1}{4}. \\
\endaligned$$
\end{proposition}

\begin{proof} By Example \ref{whittaker-example}, the module $P^\tau_1$
satisfies  that $\mathrm{wh}_{\b1}(P^\tau_1)=\C t_1^{-1}$. Since
$[\mm_n, C_{U_S}(B)]=0$, $\mathrm{wh}_{\b1}(P^\tau_1)$ is a one dimensional $C_{U_S}(B)$-module.
By direct computations, we can obtain  the action of $C_{U_S}(B)$ on $\mathrm{wh}_{\b1}(P^\tau_1)$ as follows:
$$\aligned  A_{\e_k+\e_1} t_1^{-1} &= -\frac{1}{2}t_1^{-1},
\ \ A_{\e_k+\e_l} t_1^{-1}=\frac{1}{2}(\frac{1}{2}-\delta_{kl})t_1^{-1},\\
A_{\e_k-\e_1} t_1^{-1}&=-\frac{1}{4}t_1^{-1},
\ \  A_{\e_i-\e_j} t_1^{-1} =-\frac{1}{2}t_1^{-1},\\
A_{-\e_k-\e_l} t_1^{-1}&= -t_1^{-1},
\ \  A_{-\e_1-\e_k}t_1^{-1}=(-\frac{1}{2}-n)t_1^{-1},\\
A_{-2\e_1} t_1^{-1}
&=(-n^2+\frac{1}{4})t_1^{-1}.
\endaligned$$
\end{proof}

\begin{remark}One dimensional representations of a finite $W$-algebra $W(\mg,e)$ play an  important role on the theory of $W(\mg,e)$, since the images of their annihilators under the Skryabin
equivalence are all completely prime primitive ideals of $U(\mg)$, and therefore play a key role in Joseph's theory
of Goldie rank polynomials \cite{L2}. The  completely prime primitive ideal of $U(\sp_{2n})$ corresponding to the module Proposition \ref{one-dim} is referred to as the Joseph ideal, see \cite{Pr2}.
\end{remark}

\section{Equivalence between the category of cuspidal modules over $\sp_{2n}$  and $\mathbf{fmod}$ over $W(\sp_{2n},e)$}

In this section, we use Theorem \ref{main-iso-th} to study the relationship between weight modules over $\sp_{2n}$ and
finite-dimensional modules over the minimal nilpotent finite $W$-algebra $W(\sp_{2n},e)$.
\subsection{Block decomposition}

Let $Z$ be the center of the algebra $U(\sp_{2n})$, and $Z':=\mathrm{Hom}(Z,\C)$ be set of central characters. By $\chi_{\lambda}$ we denote the  central character of the irreducible highest weight module with highest weight $\lambda$.
By the Harish-Chandra's Theorem, any central character $\chi$ is equal to $\chi_{\lambda}$ for some $\lambda$.
 A module $M$ is said to
have the  generalized central character $\chi$, if  for any $v\in M$, there is
some positive integer $k$ such that $(z-\chi(z))^kv = 0$.

Let $\mathcal{C}$ be the category of cuspidal modules over $\sp_{2n}$. For each $\ba\in \C^n$, let $\mathcal{C}_\ba$ be the full subcategory of
$\mathcal{C}$ consisting of modules $M$ such that $\mathrm{supp}(M)\subseteq [\ba]:=\ba+\mathbf{B}$, where
 $$\mathbf{B}:=\{\mathbf{b}=(b_1,\dots,b_n)\in\Z^n\mid b_1+\dots+b_n\in 2\Z\}.$$
For any $\chi\in Z'$, let $\mathcal{C}_\ba(\chi)$
be the full subcategory of $\mathcal{C}_\ba$ of modules with
generalized central character $\chi$. Since the action of $Z$ on
cuspidal modules is locally finite, we have the decomposition:
$$\mathcal{C}=\oplus_{\chi\in Z', [\ba]\in\C^n/\mathbf{B}}\mathcal{C}_\ba(\chi).$$

By Theorem \ref{W-iso}, we can use $W$ to denote both $C_{U_S}(B)$ and $W(\sp_{2n},e)$.
Let $W\text{-fmod}$ be the category of finite-dimensional
$W$-modules. By Theorems \ref{main-iso-th} and \ref{W-iso}, we can identify the center of $U_S$ with the center of $W$. For any $\chi\in Z'$,  we can also define the full subcategory $W(\chi)\text{-fmod}$ of $W\text{-fmod}$  of modules with
generalized central character $\chi$.  According to the action of center, we have the decomposition:
$W\text{-fmod}=\oplus_{\chi \in Z'} W(\chi)\text{-fmod}$.

\subsection{The equivalence between $W\text{-fmod}$ and $\mathcal{C}_{\ba}(\chi_\lambda)$}

Let $\ba\in\C^n, \lambda\in\mh_n^*$ such that  $\mathcal{C}_{\ba} (\chi_\lambda)$ is nonempty.  We define the functor $F_{\ba}: \mathcal{C}_{\ba}(\chi_\lambda) \rightarrow W(\chi_\lambda)\text{-fmod}$ such that $F_{\ba}(M)=M_{\ba}$. From $[W, \mh_n]=0$,   we have $W M_{\ba}\subset M_{\ba}$. So $M_\ba \in W(\chi_\lambda)\text{-fmod}$.
 Conversely, for  a $V\in  W(\chi_\lambda)\text{-fmod}$, let each $h_{i}$
 act  on it as the scalar $a_i$.
 Define  $$G_{\ba}(V)=\text{Ind}_{U(\mh_n)W}^{U_S}V=U_{S}\otimes_{U(\mh_n)W} V.$$ It is clear that $G_{\ba}(V)=\C[X_{2\e_1}^{\pm 1}, X_{\e_1-\e_2}^{\pm 1},\dots, X_{\e_1-\e_n}^{\pm 1}]\otimes V$.
  Thus we have a functor $G_{\ba}$ from
$W\text{-fmod}$ to the category of weight modules over $\sp_{2n}$.

\begin{lemma}\label{w-half-lemma} If  $W(\chi_\lambda)$-$\mathrm{fmod}$ is non-empty, then
$W(\chi_\lambda)$-$\mathrm{fmod}$ is equivalent to $W(\chi_{-\frac{\b1}{2}})$-$\mathrm{fmod}$, where $-\frac{\b1}{2}=(-\frac{1}{2},\dots, -\frac{1}{2})$.
\end{lemma}

\begin{proof}
Suppose that $V$  is a module in $W(\chi_\lambda)\text{-fmod}$.
Then $\sum_{\ba\in \C^n} G_{\ba} (V)$ is a uniformly bounded weight
module with the generalized central character $\chi_{\lambda}$. By Lemma 9.1 in \cite{M}, $\lambda$ satisfies the conditions:
$$\lambda_1-\lambda_2,\dots, \lambda_{n-1}-\lambda_{n},\lambda_n\in\frac{1}{2}+\Z, \lambda_{n-1}+\lambda_n\in \Z_{\geq -2}.$$
Using the translation functor (see \cite{BG}), $\mathcal{H}_{\b1}(\chi_\lambda)$ is equivalent to $\mathcal{H}_{\b1}(\chi_{-\frac{\b1}{2}})$. By the
Skryabin equivalence between $\mathcal{H}_{\b1}(\chi_\lambda)$
and $W(\chi_\lambda)\text{-fmod}$,  each  $W(\chi_\lambda)\text{-fmod}$ is equivalent to $W(\chi_{-\frac{\b1}{2}})\text{-fmod}$.
\end{proof}

\begin{lemma}\label{unique} The one dimensional module $V_{-\frac{\b1}{2}}$ defined in Proposition \ref{one-dim} is the unique simple module
in the block $W(\chi_{-\frac{\b1}{2}})$-$\mathrm{fmod}$.
\end{lemma}

\begin{proof}Since
both $P^\tau_1$ (see Example \ref{one-dim}) and the simple highest weight module $L(-\frac{\b1}{2})$
are annihilated  by $\ker \phi$, they have the same   central character $\chi_{-\frac{\b1}{2}}$.
By Theorem 6.2 in \cite{Pr2}, the central characters of simple finite-dimensional modules over $W(\sp_{2n},e)$ with different dimensions are  distinct. Then the uniqueness of $V_{-\frac{\b1}{2}}$ follows.
\end{proof}

In the following theorem, we show that  each  $W(\chi_{\lambda})\text{-fmod}$  is  equivalent to the corresponding block of the cuspidal category.

\begin{theorem}
\label{equ-cw}For any $ \lambda\in \C^n$, there exists $\ba\in\C^n$ such that
$W(\chi_\lambda)$-$\mathrm{fmod}$   is equivalent to $\mathcal{C}_{\ba}(\chi_\lambda)$.
\end{theorem}

\begin{proof}
 From the result in \cite{BKLM}, if $\mathcal{C}(\chi_{\lambda})$ is non-empty, then it is equivalent to $\mathcal{C}(\chi_{-\frac{\b1}{2}})$.
By Lemma \ref{w-half-lemma}, it suffices to consider that
$\lambda=-\frac{\b1}{2}$. Choose $\ba\in\C^n$ such that $a_1-\frac{1}{2}, \dots, a_n-\frac{1}{2}\not\in \Z$.

We will show that
$G_\ba(V)$ is a cuspidal $\sp_{2n}$-module, i.e., $G_\ba(V)\in \mathcal{C}_{\ba}(\chi_{-\frac{\b1}{2}})$. First, suppose that $V$ is a simple module in $W(\chi_{-\frac{\b1}{2}})\text{-fmod}$. By Lemma \ref{unique},
$V$ is isomorphic to the module $V_{-\frac{\b1}{2}}$ defined in Proposition \ref{one-dim}.
 Denote $X^{\bm}=X_{2\e_1}^{m_1}X_{\e_1-\e_2}^{m_2}\dots X_{\e_1-\e_n}^{m_n}$. Then
$$\aligned h_1(X^{\bm}\otimes v)&=(|\bm|+m_1+a_1)(X^{\bm}\otimes v),\\
h_i(X^{\bm}\otimes v)&=(a_i-m_i)(X^{\bm}\otimes v), i=2,\dots,n,\\
I_n(X^{\bm}\otimes v)&=(|\ba|+2m_1)(X^{\bm}\otimes v),
\endaligned$$
for any $v\in V$. By Proposition \ref{one-dim} and the formula (\ref{x-def2}), we can verify that
$$
 \aligned
  X_{\e_k+\e_1} (X^{\bm}\otimes v)
 &=(X_{\e_1-\e_k}^{-1}X_{2\e_1}A_{\e_k+\e_1}
 +h_kX_{\e_1-\e_k}^{-1}X_{2\e_1})(X^{\bm}\otimes v)\\
 &=X^{\bm-e_k+e_1}\otimes (a_k-m_k+1+A_{\e_k+\e_1})v\\
  &=(a_k-m_k+\frac{1}{2})X^{\bm-e_k+e_1}\otimes v.
 \endaligned$$
 Similarly, we obtain that
$$
\aligned
X_{\e_i-\e_j} (X^{\bm}\otimes v)&=(a_i-m_i+ \frac{1}{2}) X^{\bm-e_i+e_j}\otimes v,\\
X_{\e_k+\e_l} (X^{\bm}\otimes v)&= (a_l-m_l+\frac{1}{2})(a_k-m_k+\delta_{kl}+\frac{1}{2})X^{\bm-e_k-e_l+e_1}\otimes v,\\
X_{\e_k-\e_1}(X^{\bm}\otimes v)&=(a_k-m_k+\frac{1}{2})
 (|\bm|+a_1+m_1-\frac{1}{2}) X^{\bm-e_k}\otimes v,
 \endaligned $$
and $$
\aligned
X_{-\e_k-\e_l}(X^{\bm}\otimes v)&=-X^{\bm+e_k+e_l-e_1}\otimes  v,\\
X_{-\e_1-\e_k}(X^{\bm}\otimes v)&=(\frac{1}{2}
-|\bm|-a_1-m_1) X^{\bm+e_k-e_1}\otimes v,\\
X_{-2\e_1}(X^{\bm}\otimes v)&=- (\frac{3}{2}
-|\bm|-a_1-m_1)  (\frac{1}{2}
-|\bm|-a_1-m_1)X^{\bm-e_1}\otimes v.
\endaligned $$

By the condition on $\ba$, $G_\ba(V_{-\frac{\b1}{2}})$ is a cuspidal $\sp_{2n}$-module.
By induction on the length of $V$, we can show that $G_\ba(V)$ is a cuspidal $\sp_{2n}$-module,
 for any $V\in W(\chi_{-\frac{\b1}{2}})\text{-fmod}$.

We can check that $G_{\ba}F_{\ba}\cong \text{Id}$ and $F_{\ba}G_{\ba}\cong \text{Id}$. Therefore
$F_{\ba}, G_{\ba}$ are inverse equivalence
 between  $\mathcal{C}_{\ba}(\chi_\lambda)$ and $W(\chi_\lambda)$-$\mathrm{fmod}$.
 \end{proof}

 \begin{remark} In \cite{BKLM}, it was shown that the category of cuspidal $\sp_{2n}$-modules is semi-simple. Premet has shown that the category of finite-dimensional modules over the minimal  nilpotent
 $W(\sp_{2n},e)$  is also semi-simple, see Corollary 7.1 in \cite{Pr2}. Theorem \ref{equ-cw} explains this coincidence.
 \end{remark}

 For $\mathbf{a}=(a_1,\dots,a_n)\in \C^n$, let $N(\ba)$  be the linear span of
 $$\{t^{\mathbf{b}}:=t_1^{a_1+b_1}t_2^{a_2+b_2}\dots t_n^{a_n+b_n}\mid \mathbf{b}\in \mathbf{B}\}.$$
 Then $N(\ba)$ is an $\sp_{2n}$-module under the map $\phi$ defined in Proposition \ref{D-map}. Since
both $N(\ba)$ and the simple highest weight module $L(-\frac{\b1}{2})$
are annihilated  by $\ker \phi$, i.e., they have the same annihilator,
so the central character of $N(\ba)$ is $\chi_{-\frac{\b1}{2}}$.
If $a_i\not\in \Z$ for all $i=1,\dots, n$, the $\sp_{2n}$-module $N(\ba)$ is simple and cuspidal. It has been shown that every completely pointed simple cuspidal $\sp_{2n}$-module is isomorphic to $N(\ba)$
for $\mathbf{a}=(a_1,\dots,a_n)\in \C^n$,
 with $a_i\not\in \Z$ for all $i=1,\dots, n$, see \cite{BL}.
By Theorem \ref{equ-cw}, $G_{\ba+\frac{\b1}{2}}(V_{-\frac{1}{2}})\cong N(\ba)$.

\vspace{2mm}

\

\noindent
{\bf Acknowledgments. }This research is supported  by National Natural Science Foundation of China (Grants
12371026, 12271383).

\vspace{4mm}

 \noindent G.L.: School of Mathematics and Statistics,
and  Institute of Contemporary Mathematics,
Henan University, Kaifeng 475004, China. Email: liugenqiang@henu.edu.cn

\vspace{0.2cm}

 \noindent M.L.: School of Mathematics and Statistics,
Henan University, Kaifeng 475004, China. Email: 13849167669@163.com

\end{document}